\documentclass[12pt,a4paper]{amsart}
\usepackage{enumerate}

\usepackage{amsmath,amssymb,xspace,amsthm}
\newtheorem{theorem}{Theorem}
\newtheorem{example}[theorem]{Example}
\newtheorem{remark}[theorem]{Remark}
\newtheorem{lemma}[theorem]{Lemma}
\newtheorem{proposition}[theorem]{Proposition}
\numberwithin{equation}{section}

\newcommand{\tto}{\twoheadrightarrow}

\newcommand{\C}{\ensuremath{\mathbb C}\xspace}

\renewcommand{\phi}{\varphi}

\renewcommand{\leq}{\leqslant}
\renewcommand{\geq}{\geqslant}

\begin{document}
\title[Simple Virasoro modules]{Simple Virasoro modules which are locally finite over a positive part}
\author{Volodymyr Mazorchuk and Kaiming Zhao}
\date{\today}

\begin{abstract}
We propose a very general construction of simple Virasoro modules generalizing and including both highest weight
and Whittaker modules. This reduces the problem of classification of simple Virasoro modules which are locally
finite over a positive part to classification of simple modules over a family of finite dimensional solvable
Lie algebras. For one of these algebras all simple modules are classified by R.~Block and we
extend this classification to the next member of the family. As a result we recover many known but also construct
a lot of new simple Virasoro modules. We also propose a revision of the setup for study of Whittaker modules.
\end{abstract}
\maketitle

%
%

\section{Introduction and formulation of the results}\label{s0}

We denote by $\mathbb{N}$ the set of positive integers and by $\mathbb{Z}_+$ the set of all  non-negative integers.
For a Lie algebra $\mathfrak{a}$ we denote by $U(\mathfrak{a})$ the universal enveloping algebra of $\mathfrak{a}$.

Let $\mathfrak{V}$ denote the complex {\em Virasoro algebra}, that is the Lie algebra with basis $\{\mathtt{c},
\mathtt{l}_i:i\in\mathbb{Z}\}$ and the Lie bracket defined (for $i,j\in\mathbb{Z}$) as follows:
\begin{displaymath}
[\mathtt{l}_i,\mathtt{l}_j]=(j-i) \mathtt{l}_{i+j}+\delta_{i,-j}\frac{i^3-i}{12}\mathtt{c};\quad
[\mathtt{l}_i,\mathtt{c}]=0.
\end{displaymath}
The algebra $\mathfrak{V}$ is a very important object both in mathematics and in mathematical physics, see
for example \cite{KR,IK} and references therein. There are two classical families of simple $\mathfrak{V}$-modules:
highest weight modules (completely described in \cite{FF}) and the so-called intermediate series modules.
In \cite{Mt} it is shown that these two families exhaust all simple weight Harish-Chandra modules, that is weight
modules with finite dimensional weight spaces with respect to the Cartan subalgebra spanned by $\mathtt{l}_0$
and $\mathtt{c}$. In \cite{MZ1} it is even shown that the above modules exhaust all simple weight modules admitting
a nonzero finite dimensional weight space.

Various other families of simple $\mathfrak{V}$-modules were studied in \cite{Zh,OW,LGZ,FJK,Ya,GLZ,OW2}. These include
some simple weight modules with infinite dimensional weight spaces, various versions of Whittaker modules and some
other modules constructed using different tricks. Observe that simple highest weight modules and (all versions of)
Whittaker module have something in common: in both classes the action of a Lie subalgebra of $\mathfrak{V}$
generated by all elements $\mathtt{l}_i$, where $i$ is big enough, is locally nilpotent. The class $\mathcal{X}$
of simple $\mathfrak{V}$-modules described by the latter property is the principal object of study in the present
paper. Now we briefly describe the  main results.

Denote by $\mathfrak{V}_+$ the Lie subalgebra of $\mathfrak{V}$ spanned by all $\mathtt{l}_i$ with $i\geq 0$.
Given $N\in \mathfrak{V}_+\text{-}\mathrm{mod}$ and $\theta\in\mathbb{C}$, consider the corresponding induced module
$\mathrm{Ind}(N):=U(\mathfrak{V})\otimes_{U(\mathfrak{V}_+)}N$ and denote by $\mathrm{Ind}_{\theta}(N)$
the module $\mathrm{Ind}(N)/(\mathtt{c}-\theta)\mathrm{Ind}(N)$. The key result of this paper is:

\begin{theorem}\label{thmmain}
Assume that $N\in \mathfrak{V}_+\text{-}\mathrm{mod}$ is simple and such that there exists $k\in\mathbb{N}$
satisfying the following two conditions:
\begin{enumerate}[$($a$)$]
\item\label{thmmain.1} $\mathtt{l}_k$ acts injectively on $N$;
\item\label{thmmain.2} $\mathtt{l}_iN=0$ for all $i>k$.
\end{enumerate}
Then for any $\theta\in\mathbb{C}$ the $\mathfrak{V}$-module $\mathrm{Ind}_{\theta}(N)$ is simple.
\end{theorem}

Theorem~\ref{thmmain} is proved in Section~\ref{s1}. It gives a very general recipe for construction of simple
modules in $\mathcal{X}$. But we go even further. In Theorem~\ref{thmmain2} below we show that every simple module
in $\mathcal{X}$ is either a simple highest weight module or is obtained using the recipe from Theorem~\ref{thmmain}.
To formulate Theorem~\ref{thmmain2} we need to recall some terminology.

Recall that a module $V$ over a Lie algebra $\mathfrak{a}$ is called {\em locally finite}
provided that any $v\in V$ belongs to a finite dimensional $\mathfrak{a}$-submodule. The module is called
{\em locally nilpotent} provided that for any $v\in V$ there exists an $n\in\mathbb{N}$ such that
$a_1a_2\cdots a_n(v)=0$ for all $a_1,a_2,\dots,a_n\in\mathbf{a}$.  For $n\in\mathbb{Z}_+$, denote by
$\mathfrak{V}^{(n)}_+$ the Lie subalgebra of $\mathfrak{V}$ generated by all $\mathtt{l}_i$, $i>n$.

\begin{theorem}\label{thmmain2}
Let $L$ be a simple $\mathfrak{V}$-module. Then the following conditions are equivalent:
\begin{enumerate}[$($a$)$]
\item\label{thmmain2.1} There exists $k\in \mathbb{N}$ such that $L$ is a locally finite
$\mathfrak{V}^{(k)}_+$-module.
\item\label{thmmain2.2} There exists $n\in \mathbb{N}$ such that $L$ is a locally nilpotent
$\mathfrak{V}^{(n)}_+$-module.
\item\label{thmmain2.3} $L$ is a highest weight module or there exists $\theta\in\mathbb{C}$, $k\in \mathbb{N}$
and a simple $N\in \mathfrak{V}_+\text{-}\mathrm{mod}$ such that both conditions \eqref{thmmain.1} and
\eqref{thmmain.2} of  Theorem~\ref{thmmain} are satisfied and $L\cong \mathrm{Ind}_{\theta}(N)$.
\end{enumerate}
\end{theorem}

Theorem~\ref{thmmain2} is proved in Section~\ref{s2}. We also prove (in Subsection~\ref{s2.3}) that the condition
Theorem~\ref{thmmain2}\eqref{thmmain2.1} is equivalent to the condition that  $\mathtt{l}_i$ acts on $L$
locally finitely for all sufficiently large $i$. In Section~\ref{s3} we list many examples to which
Theorems~\ref{thmmain} and \ref{thmmain2} apply. These include all versions of Whittaker modules over
$\mathfrak{V}$ constructed in \cite{OW,LGZ,FJK}. We also construct several new families of simple
$\mathfrak{V}$-modules. For $n\in\mathbb{Z}_+$ denote by $\mathfrak{a}_n$ the  finite dimensional Lie algebra
$\mathfrak{V}_+/\mathfrak{V}^{(n)}_+$. Theorem~\ref{thmmain2} reduces classification of simple modules in
$\mathcal{X}$ to classification of simple modules over {\em finite dimensional} Lie algebras $\mathfrak{a}_n$
(for all $n$). For $n=0$ the algebra $\mathfrak{a}_0$ is commutative and its simple modules are one dimensional,
which leads exactly to simple highest weight modules. For $n=1$, all simple $\mathfrak{a}_1$-modules are
constructed in \cite{Bl}. These can be used both to recover some classes of Whittaker modules over
$\mathfrak{V}$ and to construct many new simple modules. We use the $n=1$ case to classify
(see Proposition~\ref{prop25}) all simple $\mathfrak{a}_2$-modules. This again recovers some classes of
Whittaker modules over $\mathfrak{V}$ (in particular, all modules constructed in \cite{OW}) and produces many new
simple modules. As far as we know, for $n>2$ the classification problem for simple $\mathfrak{a}_n$-modules is
still open.

We finish the paper with a revision of the general Whittaker setup from \cite{BM} in Section~\ref{s4}.

\vspace{5mm}

\noindent
{\bf Acknowledgement.} The major part of this work was done during the visit of the second author to Uppsala
in May 2012. The hospitality and financial support of Uppsala University are gratefully acknowledged. The first
author is partially supported by the Royal Swedish Academy of Sciences and by the Swedish
Research Council. The second author is partially supported by NSERC. We thank Hongjia Chen, Haijun Tan and
Xiangqian Guo for their comments on the original version of the paper.

\section{Proof of Theorem~\ref{thmmain}}\label{s1}

\subsection{An indexing set for a basis}\label{s1.1}

Denote by $\mathbb{M}$ the set of all (infinite) vectors of the form $\mathbf{i}:=(\dots,i_2,i_1)$ with coefficients
in $\mathbb{Z}_+$, satisfying the condition that the number of nonzero coefficients is finite. Let $\mathbf{0}$
denote the element $(\dots,0,0)\in\mathbb{M}$ and for $i\in\mathbb{N}$ let $\varepsilon_i$ denote the element
$(\dots,0,0,1,0,0,\dots,0)\in\mathbb{M}$, where $1$ is in the $i$'th position from the right. For
$\mathbf{i}\in\mathbb{M}$ denote by $\mathbf{d}(\mathbf{i})$ the {\em degree} of $\mathbf{i}$
defined as $\sum_{s\geq 1}i_s$ (note that the sum is finite). Denote also by $\mathbf{w}(\mathbf{i})$ the {\em weight}
of $\mathbf{i}$  defined as $\sum_{s\geq 1}s\cdot i_s$ (which is again finite).

For  $\mathbf{i}\in \mathbb{M}$ denote by $\mathtt{l}^{\mathbf{i}}$ the element
$\dots \mathtt{l}_{-3}^{i_3}\mathtt{l}_{-2}^{i_2}\mathtt{l}_{-1}^{i_1}\in U(\mathfrak{V})$ (note that the product
is, in fact, finite because of the definition of $\mathbb{M}$).
By the PBW Theorem, every element of $\mathrm{Ind}_{\theta}(N)$ can be uniquely written in the form
\begin{equation}\label{eq1}
\sum_{\mathbf{i}\in\mathbb{M}}\mathtt{l}^{\mathbf{i}} v_{\mathbf{i}},
\end{equation}
where all $v_{\mathbf{i}}\in N$ and only finitely many of  the $v_{\mathbf{i}}$'s are nonzero.
For $v\in \mathrm{Ind}_{\theta}(N)$ written in the form \eqref{eq1}, we denote by $\mathrm{supp}(v)$ the set of
all $\mathbf{i}\in \mathbb{M}$ such that $v_{\mathbf{i}}\neq 0$.

\subsection{Reverse lexicographic and principal orders}\label{s1.3}

Denote by $<$ the {\em reverse lexicographic} total order on $\mathbb{M}$, defined recursively (with respect to
the degree) as follows: $\mathbf{0}$ is the minimum element; and for different nonzero
$\mathbf{i},\mathbf{j}\in \mathbb{M}$  we have $\mathbf{i}<\mathbf{j}$  if and only if one of the
following conditions is satisfied:
\begin{itemize}
\item $\min\{s:i_s\neq 0\}>\min\{s:j_s\neq 0\}$;
\item $\min\{s:i_s\neq 0\}=\min\{s:j_s\neq 0\}=k$ and $\mathbf{i}-\varepsilon_{k}<\mathbf{j}-\varepsilon_{k}$.
\end{itemize}

Define the {\em principal} total order $\prec$ on $\mathbb{M}$ as follows: for different
$\mathbf{i},\mathbf{j}\in \mathbb{M}$ set $\mathbf{i}\prec \mathbf{j}$ if and only if one of the
following conditions is satisfied:
\begin{itemize}
\item $\mathbf{w}(\mathbf{i})<\mathbf{w}(\mathbf{j})$;
\item $\mathbf{w}(\mathbf{i})=\mathbf{w}(\mathbf{j})$ and $\mathbf{d}(\mathbf{i})<\mathbf{d}(\mathbf{j})$;
\item $\mathbf{w}(\mathbf{i})=\mathbf{w}(\mathbf{j})$ and  $\mathbf{d}(\mathbf{i})=\mathbf{d}(\mathbf{j})$,
but $\mathbf{i}< \mathbf{j}$.
\end{itemize}

\subsection{The argument}\label{s1.4}

For  a nonzero $v\in \mathrm{Ind}_{\theta}(N)$ let $\mathfrak{l}(v)$ denote the maximal (with respect to $\prec$)
element of $\mathrm{supp}(v)$, called the {\em leading term} of $v$. Let $M$ be a nonzero submodule of
$\mathrm{Ind}_{\theta}(N)$. Denote by $m$ the minimal non-negative integer for which there exist a nonzero
$v\in M$ such that $\mathbf{w}(\mathfrak{l}(v))=m$. If $m=0$, then $v$ belongs to the canonical copy
$1\otimes N$ of $N$ in $\mathrm{Ind}_{\theta}(N)$, which implies $M=\mathrm{Ind}_{\theta}(N)$ since $N$ is simple
and generates $\mathrm{Ind}_{\theta}(N)$. Our aim is to show that $m>0$ leads to a contradiction.

Assume that $m>0$ and let $v\in M$ be a nonzero element such that $\mathbf{w}(\mathfrak{l}(v))=m$.
We assume that $v$ is in the form \eqref{eq1}. Let $\mathbf{j}:=\mathfrak{l}(v)$ and set
$p:=\min\{s:j_s\neq 0\}>0$. Then the element $\mathtt{l}_{k+p}v$ belongs to $M$. To write
$\mathtt{l}_{k+p}v$ in the form \eqref{eq1} we have to move $\mathtt{l}_{k+p}$ all the way to the right
using the commutation relations. Let us for the moment assume that $\mathtt{l}_{k+p}v\neq 0$ (we will prove
it later). By Theorem~\ref{thmmain}\eqref{thmmain.2}, we have  $\mathtt{l}_{k+p}v_{\mathbf{i}}=0$ for all
$\mathbf{i}\in \mathrm{supp}(v)$, which means that every $\mathbf{i}'\in \mathrm{supp}(\mathtt{l}_{k+p}v)$
is in the support of $[\mathtt{l}_{k+p},\mathtt{l}^{\mathbf{i}}]v$ for some $\mathbf{i}\in \mathrm{supp}(v)$.
In particular, $\mathbf{w}(\mathbf{i}')<\mathbf{w}(\mathbf{i})$ and thus
$\mathbf{w}(\mathfrak{l}(\mathtt{l}_{k+p}v))<\mathbf{w}(\mathfrak{l}(v))=m$, which contradicts our choice of $m$.

It remains to show that $\mathtt{l}_{k+p}v\neq 0$. To prove this, it is enough to show that
$\mathbf{j}':=\mathbf{j}-\varepsilon_p\in \mathrm{supp}(\mathtt{l}_{k+p}v)$. By
Theorem~\ref{thmmain}\eqref{thmmain.1}, we have $\mathtt{l}_{k}v_{\mathbf{j}}\neq 0$.  This implies that
$\mathbf{j}'\in \mathrm{supp}([\mathtt{l}_{k+p},\mathtt{l}^{\mathbf{j}}]v_{\mathbf{j}})$ (to get $\mathbf{j}'$
we simply commute $\mathtt{l}_{k+p}$ with one of the $\mathtt{l}_{-p}$'s appearing in $\mathtt{l}^{\mathbf{j}}$).
In fact, it is easy to see that $\mathbf{j}'=\mathfrak{l}([\mathtt{l}_{k+p},\mathtt{l}^{\mathbf{j}}]v_{\mathbf{j}})$.
So, it remains to show that $\mathbf{j}'\not\in \mathrm{supp}([\mathtt{l}_{k+p},\mathtt{l}^{\mathbf{i}}]v_{\mathbf{i}})$
for any $\mathbf{i}\in\mathrm{supp}(v)\setminus\{\mathbf{j}\}$.

Assume first that $\mathbf{w}(\mathbf{i})<\mathbf{w}(\mathbf{j})$. From Theorem~\ref{thmmain}\eqref{thmmain.2} it
then follows that for any $\mathbf{k}\in \mathrm{supp}([\mathtt{l}_{k+p},\mathtt{l}^{\mathbf{i}}]v_{\mathbf{i}})$
we have
\begin{displaymath}
\mathbf{w}(\mathbf{k})\leq \mathbf{w}(\mathbf{i})-p<\mathbf{w}(\mathbf{j})-p=\mathbf{w}(\mathbf{j}'),
\end{displaymath}
which means that $\mathbf{k}\neq \mathbf{j}'$.

Assume now that $\mathbf{w}(\mathbf{i})=\mathbf{w}(\mathbf{j})$ but
$\mathbf{d}(\mathbf{i})<\mathbf{d}(\mathbf{j})$. If the element
$\mathbf{k}\in \mathrm{supp}([\mathtt{l}_{k+p},\mathtt{l}^{\mathbf{i}}]v_{\mathbf{i}})$
is such that $\mathbf{d}(\mathbf{k})<\mathbf{d}(\mathbf{i})$, then
\begin{displaymath}
\mathbf{d}(\mathbf{k})< \mathbf{d}(\mathbf{i})\leq \mathbf{d}(\mathbf{j})-1=\mathbf{d}(\mathbf{j}'),
\end{displaymath}
which means that $\mathbf{k}\neq \mathbf{j}'$. If
$\mathbf{k}\in \mathrm{supp}([\mathtt{l}_{k+p},\mathtt{l}^{\mathbf{i}}]v_{\mathbf{i}})$
is such that $\mathbf{d}(\mathbf{k})=\mathbf{d}(\mathbf{i})$, then such $\mathbf{k}$ can only be obtained
by commuting $\mathtt{l}_{k+p}$ with some $\mathtt{l}_{-s}$, where $s>k+p$. Therefore
\begin{displaymath}
\mathbf{w}(\mathbf{k})=\mathbf{w}(\mathbf{i})-k-p<\mathbf{w}(\mathbf{i})-p=
\mathbf{w}(\mathbf{j})-p=\mathbf{w}(\mathbf{j}')
\end{displaymath}
(note that here we used $k>0$), which again means that $\mathbf{k}\neq \mathbf{j}'$.

Finally, consider the case $\mathbf{w}(\mathbf{i})=\mathbf{w}(\mathbf{j})$ and
$\mathbf{d}(\mathbf{i})=\mathbf{d}(\mathbf{j})$. Then $\mathbf{i}<\mathbf{j}$.
Let $q:=\min\{s:i_s\neq 0\}>0$. We have either $q>p$ or $q=p$. First consider the case
$q>p$. In this case the same argument as in the previous paragraph shows that for any
$\mathbf{k}\in \mathrm{supp}([\mathtt{l}_{k+p},\mathtt{l}^{\mathbf{i}}]v_{\mathbf{i}})$
we have $\mathbf{w}(\mathbf{k})<\mathbf{w}(\mathbf{j}')$.
In the remaining case $q=p$ we get $\mathfrak{l}([\mathtt{l}_{k+p},\mathtt{l}^{\mathbf{i}}]v_{\mathbf{i}})=
\mathbf{i}-\varepsilon_q=\mathbf{i}-\varepsilon_p$. From the recursive definition of the lexicographic order
we also have $\mathbf{i}-\varepsilon_p<\mathbf{j}-\varepsilon_p=\mathbf{j}'$ and the claim follows.

\section{Proof of Theorem~\ref{thmmain2}}\label{s2}

\subsection{\eqref{thmmain2.2}$\Rightarrow$\eqref{thmmain2.1}}\label{s2.001}

Similarly to \cite[Subsection~2.3]{MZ} one shows that $L$ contains a non-zero $v$ satisfying
$\mathfrak{V}^{(n)}_+v=0$. As $L$ is generated by $v$ and has a central character, the PBW Theorem implies
that $L$ is spanned by all elements of the form
\begin{displaymath}
\mathtt{l}^{\mathbf{i}}v:=\cdots\mathtt{l}_{n-2}^{i_{n-2}}\mathtt{l}_{n-1}^{i_{n-1}}\mathtt{l}_{n}^{i_n}v
\end{displaymath}
(where only finitely many factors are different from $1$). Hence it is enough to show that each of these
vectors generates a finite dimensional $\mathfrak{V}^{(n)}_+$-submodule of $V$ (and then we can take $k=n$).
Let $I$ denote the annihilator of $\mathtt{l}^{\mathbf{i}}v$ in $U(\mathfrak{V}^{(n)}_+)$. Set
$m:=\sum_{s\leq n}|s|i_s< \infty$. Then $I$ contains all $\mathtt{l}_i$ for which  $i>m+n$ and any product of
$\mathtt{l}_j$, $n< j \leq m+n$, with at least $m+1$ factors. It follows that the codimension of $I$
in $U(\mathfrak{V}^{(n)}_+)$ is finite, which implies the claim.

\subsection{\eqref{thmmain2.3}$\Rightarrow$\eqref{thmmain2.2}}\label{s2.01}

Both highest weight modules and modules $\mathrm{Ind}_{\theta}(N)$ are generated by elements annihilated by some
$\mathfrak{V}^{(k)}_+$. Then the argument from the previous subsection shows that this
$\mathfrak{V}^{(k)}_+$ acts on these modules both locally finitely and locally nilpotently.

\subsection{\eqref{thmmain2.1}$\Rightarrow$\eqref{thmmain2.2}}\label{s2.1}

Let $V$ be a finite-dimensional $\mathfrak{V}^{(k)}_+$-module and $\mathfrak{i}\subset \mathfrak{V}^{(k)}_+$
be the kernel of the representation map. Then $\mathfrak{i}$ is an ideal of $\mathfrak{V}^{(k)}_+$ of finite
codimension. First we claim that $\mathfrak{i}$ contains $\mathtt{l}_i$ for some $i$. If not, then there exists a
minimal $m\in\mathbb{N}$ such that $\mathfrak{i}$ contains an element of the form
$a_s\mathtt{l}_s+a_{s+1}\mathtt{l}_{s+1}+\dots+a_{s+m}\mathtt{l}_{s+m}$ for some $s \in\mathbb{N}$ and complex
numbers $a_s,a_{s+1},\dots,a_{s+m}$ satisfying $a_s,a_{s+m}\neq 0$. Then $\mathfrak{i}$ contains
\begin{multline*}
[\mathtt{l}_{s},a_s\mathtt{l}_s+a_{s+1}\mathtt{l}_{s+1}+\dots+a_{s+m}\mathtt{l}_{s+m}]=\\=
a_{s+1}\mathtt{l}_{2s+1}+2a_{s+2}\mathtt{l}_{2s+2}+\dots+ma_{s+m}\mathtt{l}_{2s+m}\neq 0,
\end{multline*}
which contradicts our choice of $m$.

Next observe that $[\mathfrak{V}^{(k)}_+,\mathtt{l}_i]$ contains $\mathfrak{V}^{(2i+1)}_+$ and thus
$\mathfrak{V}^{(2i+1)}_+\subset \mathfrak{i}$. This implies that $L$ contains a nonzero $v$ such that
$\mathfrak{V}^{(2i+1)}_+v=0$. Then the argument from Subsection~\ref{s2.001} shows that $\mathfrak{V}^{(2i+1)}_+$
acts on $L$ locally nilpotently.

\subsection{\eqref{thmmain2.2}$\Rightarrow$\eqref{thmmain2.3}}\label{s2.2}

Let $k\in\mathbb{Z}_+$ be minimal such that the action of all $\mathtt{l}_i$, $i>k$, on $L$ is locally nilpotent.
Similarly to \cite[Subsection~2.3]{MZ} one shows that there exists a nonzero $v\in L$ annihilated by all
$\mathtt{l}_i$, $i>k$. If $k=0$, then $L$ is a highest weight module by \cite[Theorem~1(c)]{MZ}. Assume $k>0$.
Consider the vector space
\begin{displaymath}
N:=\{v\in L: \mathtt{l}_i v=0\text{ for all }i>k\}.
\end{displaymath}
We have $N\neq 0$ by the above agrument. As $\mathfrak{V}^{(k)}_+$ is an ideal of  $\mathfrak{V}_+$, the space $N$ is
stable under the action of $\mathfrak{V}_+$.  The central element $\mathtt{c}$ acts on any simple module as a scalar
(\cite[Proposition~2.6.8]{Di}), let $\theta$ be the corresponding scalar for $L$.
As $L$ is simple, it is generated by $N$ and hence the canonical map $\mathrm{Ind}_{\theta}(N)\to L$
sending $1\otimes x$ to $x$ for every $x\in N$ is surjective.

\begin{lemma}\label{lem21}
The canonical map $\mathrm{Ind}_{\theta}(N)\to L$ is an isomorphism.
\end{lemma}

\begin{proof}
Assume that this is not the case and let $K\neq 0$ be the kernel of the canonical map.
Choose nonzero $v\in K$, written in the form \eqref{eq1}, such that $\mathbf{w}(\mathfrak{l}(v))=m$ is minimal possible.
Note that $m>0$ as the canonical map is bijective for $v$ satisfying $\mathbf{w}(\mathfrak{l}(v))=0$.
Since $K$ is a submodule, we have $\mathtt{l}_{k+p}v\in K$ for all $p\in\mathbb{N}$. Choosing $p$ as in
Subsection~\ref{s1.4} and using the proof from there (which starts in the second paragraph and does not use
simplicity of $N$), we get that the element $\mathtt{l}_{k+p}v$ from $K$ is nonzero and its leading coefficient
has strictly smaller weight, a contradiction.
\end{proof}

Let us now show that $N$ is simple. Assume that this is not the case and let $N'\subset N$ be a proper submodule.
Then $\mathrm{Ind}_{\theta}(N')$ is a proper submodule of $\mathrm{Ind}_{\theta}(N)$ by the PBW theorem, a
contradiction.

To complete the proof of Theorem~\ref{thmmain2} it remains to show that $\mathtt{l}_k$ acts injectively on
$N$. Assume that this is not the case. Then there is a nonzero $v\in N$ such that $\mathtt{l}_k v=0$.
As all $\mathtt{l}_i$, $i>k$, act on  $L$ locally nilpotently and $L$ is simple, it follows easily that the action of
$\mathtt{l}_k$ on $L$ is locally nilpotent as well. This contradicts our choice of $k$ and completes the proof.

\subsection{Reformulation of condition \eqref{thmmain2.1}}\label{s2.3}

As mentioned in the previous subsection, condition \eqref{thmmain2.2} of Theorem~\ref{thmmain2} is equivalent to
the condition that there exist $k\in\mathbb{N}$ such that for every $i>k$ the action of $\mathtt{l}_i$ on $L$ is
locally nilpotent (confer \cite[Theorem~1.1(c)]{MZ}). The aim of this subsection is to prove the following
similar reformulation for condition Theorem~\ref{thmmain2}\eqref{thmmain2.1}:

\begin{proposition}\label{prop35}
Condition \eqref{thmmain2.1} of Theorem~\ref{thmmain2} is equivalent to:
\begin{equation}
\label{eq5}
\exists k\in\mathbb{N}
\text{ such that } \forall i>k \text{ the action of } \mathtt{l}_i \text{ on } L \text{ is locally finite.}
\end{equation}
\end{proposition}

That Theorem~\ref{thmmain2}\eqref{thmmain2.1} implies \eqref{eq5} is obvious so we only have to establish the
reverse inclusion. To do this we will need several lemmata.

\begin{lemma}\label{lem36}
Let $k\in\mathbb{N}$ and $V$ be a simple $\mathfrak{V}$-module containing a non-zero finite dimensional
$\mathfrak{V}_+^{(k)}$-submodule $X$. Then $V$ is a locally finite $\mathfrak{V}^{(k)}_+$-module.
\end{lemma}

\begin{proof}
This is \cite[Theorem~6]{BM}.
\end{proof}

\begin{lemma}\label{lem37}
Let $k\in\mathbb{N}$. Assume that $V$ is a $\mathfrak{V}_+^{(k)}$-module on which $\mathtt{l}_{k+1}$
acts locally finitely. Let further $v\in V$ and $\lambda\in\mathbb{C}$ be
such that $v\neq 0$ and $\mathtt{l}_{k+1}v=\lambda v$. Then $\dim_{\mathbb{C}} \mathfrak{V}_+^{(k)}v<\infty$.
\end{lemma}

\begin{proof}
As the action of $\mathtt{l}_{k+1}$ on $V$ is locally finite, for any $i$ from the set $\{k+2,k+3,\dots,2k+2\}$ there
is some $m_i\in\mathbb{N}$ such that the elements $\mathtt{l}_iv$, $(\mathtt{l}_{k+1}-\lambda)\mathtt{l}_iv$,
$(\mathtt{l}_{k+1}-\lambda)^2\mathtt{l}_iv$,\dots,
$(\mathtt{l}_{k+1}-\lambda)^{m_i}\mathtt{l}_iv$ are linearly dependent. By induction on $s$ one shows that
$(\mathtt{l}_{k+1}-\lambda)^s\mathtt{l}_iv$ equals $\mathtt{l}_{i+s(k+1)}v$ up to a scalar, which implies
existence of a relation of the form
\begin{displaymath}
\sum_{s=0}^{m_i}\alpha_{i,s}\mathtt{l}_{i+s(k+1)}v=0,
\end{displaymath}
where $\alpha_{i,s}\in \mathbb{C}$ and not all of them are zero. Taking the commutator with
$\mathtt{l}_{k+1}-\lambda$ gives
\begin{displaymath}
\sum_{s=1}^{m_i+1}\alpha^{(1)}_{i,s}\mathtt{l}_{i+s(k+1)}v=0,
\end{displaymath}
where $\alpha^{(1)}_{i,s}\in \mathbb{C}$ and not all of them are zero. This can be continued inductively.
Let $m:=\max\{m_i:i\in\{k+2,k+3,\dots,2k+2\}\}$. It follows that $\mathfrak{V}_+^{(k)}v$ coincides with the
linear span of $\mathtt{l}_{j}v$, $j\leq (m+3)(k+1)$.
\end{proof}

\begin{proof}[Proof of Proposition~\ref{prop35}.]
As the action of $\mathtt{l}_{k+1}$ on $L$ is locally finite, there is $\lambda\in\mathbb{C}$ and a non-zero
element $v\in L$ such that $\mathtt{l}_{k+1}v=\lambda v$. By Lemma~\ref{lem37}, there is $m\in\mathbb{N}$
such that $V_{k+1}:=\mathfrak{V}_+^{(k)}v$ is spanned by $v$, $\mathtt{l}_{k+2}v$, $\mathtt{l}_{k+3}v$,\dots,
$\mathtt{l}_{m}v$. For $i\in\{k+2,k+3,\dots,m\}$ define now the finite dimensional space $V_{i}$ and
a positive integer $d_i$ inductively as follows: if $V_{i-1}$ is defined, then $d_i$ is the degree of
some nonzero polynomial $f_{i}(\mathtt{l}_i)$ which annihilates $V_{i-1}$ (this is well-defined as
$V_{i-1}$ is finite dimensional and $\mathtt{l}_i$ acts on $L$ locally finitely); if $V_{i-1}$ and
$d_i$ are defined, then $V_{i}$ is the linear span of $V_{i-1}$, $\mathtt{l}_iV_{i-1}$, $\mathtt{l}^2_iV_{i-1}$,\dots,
$\mathtt{l}^{d_i}_iV_{i-1}$, which is finite dimensional.

By Lemma~\ref{lem36}, to prove Proposition~\ref{prop35} it is enough to show that $L$ contains a nonzero
finite dimensional $\mathfrak{V}_+^{(k)}$-submodule. For this it is enough to show that the finite dimensional
vector space $V_{m}$ defined in the previous paragraph is, in fact, a $\mathfrak{V}_+^{(k)}$-submodule. This vector
space is spanned over $\mathbb{C}$ by elements of the form
\begin{displaymath}
\mathtt{l}^{\mathbf{i}}v:=\mathtt{l}_{m}^{i_m}\mathtt{l}_{m-1}^{i_{m-1}}\cdots
\mathtt{l}_{k+2}^{i_{k+2}}v,\quad 0\leq i_s\leq d_s\text{ for all }s\in\{k+2,k+3,\dots,m\}.
\end{displaymath}
We will show, by induction on the weight and the degree of $\mathbf{i}$, that for any
$\mathtt{l}_s$, $s>k$, we have $\mathtt{l}_s \mathtt{l}^{\mathbf{i}}v\in V_m$. The basis of the induction
($\mathbf{i}=\mathbf{0}$) follows from definitions and Lemma~\ref{lem37}.

To prove the induction step we have to move $\mathtt{l}_s$ in $\mathtt{l}_s \mathtt{l}^{\mathbf{i}}v$ all the
way to the right using the commutation relations. Consider first the case $s>m$. In this case any
new element $\mathtt{l}_p$ which appears as the result of commutation satisfies $p>m$ and any commutation
decreases the degree of the monomial. When such $\mathtt{l}_p$ is moved all the way to the right and applied
to $v$, we use Lemma~\ref{lem37} which expresses $\mathtt{l}_pv$ as a linear combination of $\mathtt{l}_qv$
with $q<p$. This keeps the degree but decreases the weight. Therefore in both cases we may apply the
inductive assumption.

If $s\leq m$, then two types of new elements $\mathtt{l}_p$ may appear. Those for which $p>m$ are dealt with
as described in the previous paragraph. Those for which $p\leq m$ should be commuted to their natural place in the
monomial. If this results in the fact that the corresponding degree of $\mathtt{l}_p$ will exceed $d_p$,
then we use that $f_{p}(\mathtt{l}_p)$ annihilates $V_{p-1}$ which allows us to write this element as a linear
combination of elements both of smaller weight and degree. This justifies the induction and completes the proof.
\end{proof}

\section{Old and new simple Virasoro modules}\label{s3}

\subsection{Highest weight modules}\label{s3.1}

Let $\mathfrak{h}$ be the Cartan subalgebra of $\mathfrak{V}$, spanned by $\mathtt{l}_0$ and $\mathtt{c}$.
For $\lambda\in\mathfrak{h}^*$ we have the {\em Verma module}
$M(\lambda):=U(\mathfrak{V})\otimes_{U(\mathfrak{h}+\mathfrak{V}_+)}\mathbb{C}_{\lambda}$, where
$\mathtt{l}_i\mathbb{C}_{\lambda}=0$ for $i>0$, while $\mathtt{l}_0$ and $\mathtt{c}$ act on $\mathbb{C}_{\lambda}$
via scalars $\lambda(\mathtt{l}_0)$ and $\lambda(\mathtt{c})$, respectively. The module $M(\lambda)$ has the
unique simple top $L(\lambda)$, the unique (up to isomorphism) {\em simple highest weight module} with highest
weight $\lambda$. In general, $M(\lambda)\not\simeq L(\lambda)$, see e.g. \cite{IK} for details. These modules
correspond to the case $k=0$ in Theorem~\ref{thmmain2}.

\subsection{Whittaker modules of Ondrus and Wiesner}\label{s3.2}

Consider a nonzero  $\lambda:=(\lambda_1,\lambda_2)\in\mathbb{C}^2$. Denote by $N_{\lambda}$ the $\mathfrak{V}_+$-module
$U(\mathfrak{V}_+)/I$, where $I$ is the left ideal generated by $\mathtt{l}_1-\lambda_1$, $\mathtt{l}_2-\lambda_2$,
$\mathtt{l}_3$, $\mathtt{l}_4$,\dots. By the PBW Theorem, $N_{\lambda}\cong\mathbb{C}[\mathtt{l}_0]$ as a
$\mathbb{C}[\mathtt{l}_0]$-module, $\mathtt{l}_iN_{\lambda}=0$ for $i>2$, and for $i=1,2$ we have
$\mathtt{l}_if(\mathtt{l}_0)=f(\mathtt{l}_0-i)\lambda_i$, $f(\mathtt{l}_0)\in \mathbb{C}[\mathtt{l}_0]$.
As $\lambda\neq 0$, taking linear combinations of
$f(\mathtt{l}_0)$ and $f(\mathtt{l}_0-i)\lambda_i$ over $\mathbb{C}$ one can always reduce the degree, which implies
that $N_{\lambda}$ is simple. The module $N_{\lambda}$ obviously satisfies the conditions of
Theorem~\ref{thmmain} (with $k\in\{1,2\}$). Hence, by Theorem~\ref{thmmain}, we obtain the corresponding
simple induced $\mathfrak{V}$-module $\mathrm{Ind}_{\theta}(N_{\lambda})$. These are exactly the
Whittaker modules over $\mathfrak{V}$ constructed in \cite{OW}.

\subsection{Whittaker modules of L{\"u}, Guo and Zhao}\label{s3.3}

Consider some $k\in\mathbb{N}$ and set $r:=\lceil\frac{k}{2}\rceil-1$. Choose
$\lambda=(\lambda_{r+1},\lambda_{r+2},\dots, \lambda_{k})\in\mathbb{C}^{k-r}$ such that $\lambda_{k}\neq 0$. Denote
by $K_{\lambda}$ the $\mathfrak{V}_+$-module $U(\mathfrak{V}_+)/I$, where $I$ is the left ideal generated by
\begin{displaymath}
\mathtt{l}_{r+1}-\lambda_{r+1}, \mathtt{l}_{r+2}-\lambda_{r+2},\dots,
\mathtt{l}_{k}-\lambda_{k},  \mathtt{l}_{k+1}, \mathtt{l}_{k+2},\dots.
\end{displaymath}

\begin{lemma}\label{lem31}
The  $\mathfrak{V}_+$-module $K_{\lambda}$ is simple.
\end{lemma}

\begin{proof}
By the PBW Theorem, a basis of $K_{\lambda}$ is given by the images of
$\mathtt{l}_0^{i_0}\mathtt{l}_1^{i_1}\dots\mathtt{l}_r^{i_r}=:\mathtt{l}^\mathbf{i}$, where
$\mathbf{i}:=(i_0,i_1,\dots,i_r)\in\mathbb{Z}_+^{r+1}$. Let $<$ be the lexicographic order on $\mathbb{Z}_+^{r+1}$,
that is $\mathbf{0}=(0,0,\dots,0)$ is the minimum element and $\mathbf{i}<\mathbf{j}$ provided that
\begin{itemize}
\item $\min\{s:i_s\neq 0\}>\min\{s:j_s\neq 0\}$ or
\item $\min\{s:i_s\neq 0\}=\min\{s:j_s\neq 0\}=m$ and $\mathbf{i}-\varepsilon_m<\mathbf{j}-\varepsilon_m$.
\end{itemize}
Let $v\in K_{\lambda}$ be a nonzero linear combination of basis elements and $\mathbf{i}$ be the highest
term in $\mathrm{supp}(v)$ with respect to $<$. Set $p:=\min\{s:i_s\neq 0\}$. Similarly to Subsection~\ref{s1.4}
one shows that $(\mathtt{l}_{k-p}-\lambda_{k-p})v$ is nonzero with highest term $\mathbf{i}-\varepsilon_p$. Repeating
this process inductively (with respect to the degree of $\mathbf{i}$) we get that the submodule in $K_{\lambda}$
generated by $v$ contains an element with support $\{\mathbf{0}\}$. Any such element is the image of a
nonzero constant and hence generates $K_{\lambda}$. The claim follows.
\end{proof}

The module $K_{\lambda}$ obviously satisfies the conditions of Theorem~\ref{thmmain} (with the same $k$).
Hence, by Theorem~\ref{thmmain}, we obtain the corresponding simple induced $\mathfrak{V}$-module
$\mathrm{Ind}_{\theta}(K_{\lambda})$. These are exactly the
Whittaker modules over $\mathfrak{V}$ constructed in \cite[Theorem~6]{LGZ}.

\subsection{Whittaker modules of Felinska, Jaskolski and Kosztolowicz}\label{s3.4}

Consider some $k\in\mathbb{N}$, $k>1$, and choose $\lambda=(\lambda_{1},\lambda_{k})\in\mathbb{C}^{2}$ such that
$\lambda_{k}\neq 0$. Denote by $\tilde G_{\lambda}$ the $\mathfrak{V}_+$-module $U(\mathfrak{V}_+)/I$, where $I$ is
the left ideal generated by $\mathtt{l}_{1}-\lambda_{1}$, $\mathtt{l}_{k}-\lambda_{k}$, $\mathtt{l}_{k+1}$,
$\mathtt{l}_{k+2}$,\dots. The induced Virasoro module $\mathrm{Ind}_{\theta}(\tilde G_{\lambda})$ was
constructed in \cite[Section~III]{FJK}, where it was claimed that it is simple. Unfortunately, this is not the
case if $k\ge 4$ as we explain below:

By the PBW Theorem, a basis of $\tilde G_{\lambda}$ is given by the images of
$\mathtt{l}_0^{i_0}\mathtt{l}_2^{i_2}\dots\mathtt{l}_{k-1}^{i_{k-1}}=:\mathtt{l}^\mathbf{i}$,
where $\mathbf{i}:=(i_0,i_2,\dots,i_{k-1})\in\mathbb{Z}_+^{k-1}$. It
is easy to verify that the following vectors in $\tilde G_{\lambda}$
generate a proper nonzero submodule $I'$:
\begin{equation}\label{rel}a_2\mathtt{l}_{2}-
\mathtt{l}_{3}\mathtt{l}_{k-1},\,\, a_3\mathtt{l}_{3}-
\mathtt{l}_{4}\mathtt{l}_{k-1} ,\,\,\dots,
\,\,a_{k-2}\mathtt{l}_{k-2}-
\mathtt{l}_{k-1}\mathtt{l}_{k-1},
\end{equation}
where
\begin{displaymath}
(j-1)a_j=ja_{j+1}+(k-2)\lambda_k,\,\, j=2, 3,\dots, k-3;\quad
a_{k-2}=\frac{2(k-2)\lambda_k}{(k-3)}.
\end{displaymath}
Set $G_\lambda=\tilde G_\lambda/I'$. We can now prove the following result.

\begin{lemma}\label{lem32}
\begin{enumerate}[$($a$)$]
\item\label{lem32.1}
For $k=3$, the  $\mathfrak{V}_+$-module $\tilde G_{\lambda}$ is simple.
\item\label{lem32.2}  For $k=4$,  the  $\mathfrak{V}_+$-module $G_{\lambda}$ is simple.
\end{enumerate}
\end{lemma}

\begin{proof}
By the PBW Theorem, a basis of $\tilde G_{\lambda}$ is given by the images of
$\mathtt{l}_0^{i_0}\mathtt{l}_2^{i_2} =:\mathtt{l}^\mathbf{i}$, where $\mathbf{i}:=(i_0,i_2 )\in\mathbb{Z}_+^{2}$.
Let $<$ be the lexicographic order on $\mathbb{Z}_+^{2}$ defined as in the proof of Lemma~\ref{lem31}. Let
$v\in \tilde G_{\lambda}$ be a nonzero linear combination of basis elements and $\mathbf{i}$ be the highest term
in $\mathrm{supp}(v)$ with respect to $<$. A direct calculation shows that the
leading term  of $v'=(\mathtt{l}_{3}-\lambda_3)^{i_0}v$ has the
form $m\varepsilon_2$ for some $m\in\mathbb{Z}_+$. Then $(\mathtt{l}_{1}-\lambda_1)^{m}v'$ is a nonzero constant
and hence generates $\tilde G_{\lambda}$. Claim \eqref{lem32.1} follows.

By the PBW Theorem and (\ref{rel}), a basis of $G_{\lambda}$ is given by the images of
\begin{equation}\label{rel2}
\mathtt{l}_0^{i_0}\mathtt{l}_2^{i_2}\mathtt{l}_{3}^0
=:\mathtt{l}^\mathbf{i},\,\,\,\,\,\,
\mathtt{l}^\mathbf{i}\mathtt{l}_{3},
\end{equation} where
$\mathbf{i}:=(i_0,i_2,0)\in\mathbb{Z}_+^{3}$. Let $<$ be the total
order on $\mathbb{Z}_+^{3}$, that is $\mathbf{0}=(0,0,0)$ is the
minimum element and $\mathbf{i}<\mathbf{j}$ provided that
\begin{itemize}
\item $\min\{s:i_s\neq 0\}>\min\{s:j_s\neq 0\}$ or
\item $\min\{s:i_s\neq 0\}=\min\{s:j_s\neq 0\}=m$ and $\mathbf{i}-\varepsilon_m<\mathbf{j}-\varepsilon_m$.
\end{itemize}
Let $v\in G_{\lambda}$ be a nonzero linear combination of basis elements and $\mathbf{i}$ be the
highest term in $\mathrm{supp}(v)$ with respect to $<$. Note that $\mathbf{i}$ has a special form because of \eqref{rel2}. The leading term $\mathbf{i}'$ of $v'=(\mathtt{l}_{4}-\lambda_4)^{i_0}v$ has the property that $i'_0=0$.
Set
\begin{displaymath}
p:=\max\{s>0:i'_s\neq 0\}.
\end{displaymath}
Similarly to Subsection~\ref{s1.4} one shows that $(\mathtt{l}_{1}-\lambda_1)v$ is nonzero with highest term
\begin{displaymath}
\begin{cases}
\mathbf{i}'-\varepsilon_2+\varepsilon_{3}, & p=2;\\
\mathbf{i}'-\varepsilon_{3}, & \text{otherwise}.
\end{cases}
\end{displaymath}
Repeating this process inductively (with respect to both the degree and the weight of $\mathbf{i}$) we get that the
submodule in $G_{\lambda}$ generated by $v$ contains the generator $1$ of $G_\lambda$. Claim \eqref{lem32.2} follows.
\end{proof}

If $k>4$, one can show that the modules $G_{\lambda}$ are still reducible.
In the cases described by Lemma~\ref{lem32}, the module $G_{\lambda}$ (resp. $\tilde{G}_{\lambda}$) obviously
satisfies the conditions of Theorem~\ref{thmmain} giving us the corresponding simple induced modules.

\subsection{Other Whittaker modules}\label{s3.9}

Fix $k\in\mathbb{N}$. Choose a pair $(S,\lambda)$, where $S\subset\{1,2,\dots,k\}$ and
$\lambda=(\lambda_i)_{i\in S}\in\mathbb{C}^{|S|}$, such that the following conditions are satisfied:
\begin{enumerate}[(I)]
\item\label{cond1} $k\in S$ and $\lambda_k\neq 0$;
\item\label{cond2} for all $i,j\in S$, $i\neq j$, we either have $i+j>k$ or $i+j\in S$ and $\lambda_{i+j}=0$;
\item\label{cond3}
for all $j\in \{1,2,\dots,k\}\setminus S$, we have $k-j\in S$.
\end{enumerate}
One example
is $S=\{\lceil\frac{k}{2}\rceil,\lceil\frac{k}{2}\rceil+1,\dots,k\}$
and any $\lambda$ with  $\lambda_k\neq 0$ (confer
Subsection~\ref{s3.3}). Another example is $S=\{2,4,5\}$ for $k=5$
and any $\lambda$ with  $\lambda_5\neq 0$. Our final example is
$S=\{3,4,6,7,8\}$ for $k=8$ and any $\lambda$ with $\lambda_8\neq 0$
and $\lambda_7=0$ (note that here we have $3+4\in S$). Denote by
$Q_{\lambda}$ the $\mathfrak{V}_+$-module $U(\mathfrak{V}_+)/I$,
where $I$ is the left ideal generated by
$\mathtt{l}_{i}-\lambda_{i}$, $i\in S$, and $\mathtt{l}_{k+1}$,
$\mathtt{l}_{k+2}$,\dots. Condition \eqref{cond2} guarantees that
the module $Q_{\lambda}$ is nonzero.

\begin{lemma}\label{lem39}
The  $\mathfrak{V}_+$-module $Q_{\lambda}$ is simple.
\end{lemma}

\begin{proof}
The proof is similar to that of Lemma 7. We omit the details.
\end{proof}

The module $Q_{\lambda}$ obviously satisfies the conditions of
Theorem~\ref{thmmain} (with the same $k$ because of condition
\eqref{cond1}). Hence, by Theorem~\ref{thmmain}, we obtain the
corresponding simple induced $\mathfrak{V}$-module
$\mathrm{Ind}_{\theta}(Q_{\lambda})$. As far as we can judge, if $S$
is not of the form
$\{\lceil\frac{k}{2}\rceil,\lceil\frac{k}{2}\rceil+1,\dots,k\}$,
then the simple $\mathfrak{V}$-module
$\mathrm{Ind}_{\theta}(Q_{\lambda})$ is new.

The following remark is suggested by Hongjia Chen.

\begin{remark}\label{remerre}
If $S\subset\{1,2,\dots,k\}$ does not satisfy \eqref{cond3}, then there exists $\lambda \in \C^{|S|}$ 
satisfying both \eqref{cond1} and \eqref{cond2} such that $Q_{\lambda}$ is reducible as a $\mathfrak{V}_+$-module.
\end{remark}

\begin{proof} 
Take a nonzero $\lambda $ such that $\lambda_i = 0$ for all $i \in S \setminus\{k\}$, we will show that the 
corresponding $Q_{\lambda}$ is reducible. Let $r$ be the minimal integer such that 
$r, k-r \in \{1,\dots,k\} \setminus S$. For any $i \in S$ (note that $i\ne r$), the equality 
$(\mathtt{l}_i - \lambda_i)\mathtt{l}_{k-r}v = [\mathtt{l}_i,\mathtt{l}_{k-r}]v \neq 0$ implies 
$k - (r - i) \in \{1,\dots,k\} \setminus S$. By the assumption on $r$, this implies $r - i \in S$. 
If $r-i \neq i$, then we have $r = (r-i)+i \in S$ using  \eqref{cond2}, which is a contradiction. Thus $i=r/2$.

{\bf Case 1:}  $\frac{r}{2} \notin S$. We see that $(\mathtt{l}_i -\lambda_i)\mathtt{l}_{k-r}v = 0$ for all 
$i \in S$. So $\mathtt{l}_{k-r}v $ generates a proper nonzero submodule of $Q_\lambda$.

{\bf Case 2:}   Now $s = \frac{r}{2} \in S$. Then $(\mathtt{l}_i -\lambda_i)\mathtt{l}_{k-2s}v = 0$ for all 
$i \in S$ with $i \neq s$. Let $w = (k-3s)\mathtt{l}_{k-s}^2v - 2 \lambda_k (k-2s)\mathtt{l}_{k-2s}v$. 
It is easy to check that $(\mathtt{l}_s - \lambda_s)w=0$.

For all $i \in S$ with $i \neq s$, we see that 
\begin{displaymath}
(\mathtt{l}_i - \lambda_i)w=(k-3s)(\mathtt{l}_i -\lambda_i)\mathtt{l}_{k-s}^2v = 2(k-3s)(k-s-i)
\mathtt{l}_{k-s}\mathtt{l}_{k-s+i}v 
\end{displaymath}
since $2k-2s+i>k$. If $\mathtt{l}_{k-s+i}v\neq 0$, then $k-s+i \in \{1,\dots,k\} \setminus S$ and thus 
$s-i \in S$ since $s<r$. Consequently $2s = s + (s-i) +i \in S$ which is a contradiction. Thus
$(\mathtt{l}_i - \lambda_i)\mathtt{l}_{k-2s}v = 0$ for all $i \in S$, and $w$ generates a proper nonzero 
submodule of $Q_\lambda$.
\end{proof}

\subsection{Block modules}\label{s3.5}

Denote by $\mathfrak{b}$ the two-dimensional Lie algebra $\mathfrak{V}_+/\mathfrak{V}_+^{(1)}$. Let
$N$ be a simple infinite dimensional $\mathfrak{b}$-module. Then $N$ has the induced structure of a
simple $\mathfrak{V}_+$-module, which obviously satisfies the conditions of Theorem~\ref{thmmain} (with $k=1$).
Hence, by Theorem~\ref{thmmain}, we obtain the corresponding simple induced $\mathfrak{V}$-module
$\mathrm{Ind}_{\theta}(N)$.
A classification of simple $\mathfrak{b}$-modules is obtained by Block in \cite[Section~6]{Bl} (this includes a very
explicit family of simple modules proposed by Arnal and Pinczon in \cite[Section~5]{AP}). As far as we can judge,
if we leave out weight and Whittaker modules (a ``negligible'' set in Block's classification), all other
simple $\mathfrak{V}$-modules obtained in this way are new.

\subsection{Block modules for $k=2$}\label{s3.6}

Denote by $\mathfrak{c}$ the three-dimensional Lie algebra $\mathfrak{V}_+/\mathfrak{V}_+^{(2)}$. Note that
the algebra $\mathfrak{b}$ from the previous subsection is a subalgebra of $\mathfrak{c}$. Abusing notation,
we identify $\mathtt{l}_i$, $i=0,1,2$, with their images in both $\mathfrak{b}$ and $\mathfrak{c}$.
For any $\lambda\in \mathbb{C}$, mapping
\begin{displaymath}
\mathtt{l}_0\to \mathtt{l}_0,\quad \mathtt{l}_1\to \mathtt{l}_1,\quad \mathtt{l}_2\to \lambda\mathtt{l}_1^2
\end{displaymath}
extends to a surjective map $\varphi_{\lambda}:U(\mathfrak{c})\tto U(\mathfrak{b})$. Let
\begin{displaymath}
\overline{\varphi}_{\lambda}:\mathfrak{b}\text{-}\mathrm{mod}\to \mathfrak{c}\text{-}\mathrm{mod}
\end{displaymath}
denote the induced pullback functor. As $\varphi_{\lambda}$ is an epimorphism,
$\overline{\varphi}_{\lambda}$ maps simple modules
to simple modules. Let $L$ be a simple $\mathfrak{b}$-module (from \cite[Section~6]{Bl}) and
$N=\overline{\varphi}_{\lambda}(L)$, where $\lambda\neq 0$ (it is easy to see that different $\lambda$ give
non-isomorphism modules). Then $N$ has the induced structure of a
simple $\mathfrak{V}_+$-module, which obviously satisfies the conditions of Theorem~\ref{thmmain} (with $k=2$).
Hence, by Theorem~\ref{thmmain},  we obtain the corresponding simple induced $\mathfrak{V}$-module
$\mathrm{Ind}_{\theta}(N)$. As far as we can judge, most  of these simple $\mathfrak{V}$-modules are new.

Let $\psi:\mathfrak{c}\to \mathfrak{b}$ be the unique Lie algebra epimorphism which sends
\begin{displaymath}
\mathtt{l}_0\to \frac{1}{2}\mathtt{l}_0,\quad \mathtt{l}_1\to 0,\quad \mathtt{l}_2\to \mathtt{l}_1.
\end{displaymath}
Let $\overline{\psi}:\mathfrak{b}\text{-}\mathrm{mod}\to \mathfrak{c}\text{-}\mathrm{mod}$ be the induced
pullback functor, which again sends simple modules to simple modules. Similarly to the previous paragraph, for any
simple $\mathfrak{b}$-module  $L$ and any $\theta\in\mathbb{C}$ we get the simple module
$\mathrm{Ind}_{\theta}(\overline{\psi}(L))$. Moreover, we have the following:

\begin{proposition}\label{prop25}
Every simple $\mathfrak{c}$-module $N$ is isomorphic either to $\overline{\varphi}_{\lambda}(L)$  for some
$\lambda\in\mathbb{C}$ and some simple $\mathfrak{b}$-module $L$ or to $\overline{\psi}(L)$ for some
simple $\mathfrak{b}$-module $L$.
\end{proposition}

\begin{proof}
First we assume that $N$ contains a nonzero $v$ such that $\mathtt{l}_2v=0$. Then every element in $N$ can be
written in the form $\sum_{i,j\geq 0}\alpha_{i,j}\mathtt{l}_0^i\mathtt{l}_1^j v$ where
$\alpha_{i,j}\in\mathbb{C}$ and we have
\begin{displaymath}
\mathtt{l}_2\left(\sum_{i,j\geq 0}\alpha_{i,j}\mathtt{l}_0^i\mathtt{l}_1^j v\right)=
\sum_{i,j\geq 0}\alpha_{i,j}(\mathtt{l}_0-2)^i\mathtt{l}_1^j \mathtt{l}_2 v=0,
\end{displaymath}
which means that $\mathtt{l}_2 N=0$. It follows that the restriction $L$ of $N$ to $\mathfrak{b}$
is simple and $N\cong \overline{\varphi}_{0}(L)$.

Assume now that $N$ contains a nonzero $v$ such that $\mathtt{l}_1v=0$. The same arguments as in the previous
paragraph show that $\mathtt{l}_1N=0$, which implies that $N\cong \overline{\psi}(L)$
for some simple  $\mathfrak{b}$-module $L$.

Finally, assume that both $\mathtt{l}_2$ and $\mathtt{l}_1$ act injectively on $N$.
Let $A$ denote the localization of $U(\mathfrak{c})$ with respect to powers of $\mathtt{l}_2$.
The element $\mathtt{l}_1^2\mathtt{l}_2^{-1}$ belongs to the center of $A$ and is nonzero. If $N$ is a simple
$\mathfrak{c}$-module on which $\mathtt{l}_2$ acts injectively, then $N$ localizes to a simple
$A$-module $N'$. As $A$ is finitely generated (it is generated by  $\mathtt{l}_0$, $\mathtt{l}_1$,
$\mathtt{l}_2$ and $\mathtt{l}_2^{-1}$), the central element $\mathtt{l}_1^2\mathtt{l}_2^{-1}$ must act
as some scalar, say $\lambda\in\mathbb{C}\setminus\{0\}$, on $N'$ (confer the proof of \cite[Theorem~4.7]{Ma}).
This means that $\mathtt{l}_1^2-\lambda\mathtt{l}_2$ annihilates $N$. Again, it follows that the
restriction $L$ of $N$ to $\mathfrak{b}$  is simple and $N\cong \overline{\varphi}_{\mu}(L)$,
where $\mu=\lambda^{-1}$.
\end{proof}

\section{Revision of the Whittaker setup}\label{s4}

In this section we would like to take the opportunity to extend and correct the Whittaker setup proposed in \cite{BM}.

\subsection{Whittaker pairs and Whittaker modules}\label{s4.1}

Let $\mathfrak{n}$ be a Lie algebra. Set $\mathfrak{n}_0:=\mathfrak{n}$ and define recursively
$\mathfrak{n}_{i}:=[\mathfrak{n}_{i-1},\mathfrak{n}]$, $i\in\mathbb{N}$. We will
call the algebra $\mathfrak{n}$ {\em quasi-nilpotent} provided that the descending chain
$\mathfrak{n}_0\supset \mathfrak{n}_1\supset\dots$ of ideal has zero intersection and each $\mathfrak{n}_{i}$
has finite codimension in $\mathfrak{n}$ (confer \cite[3.1]{BM} where the last condition is missing).
Denote by $\mathfrak{F}_{\mathfrak{n}}$ the category of all finite dimensional $\mathfrak{n}$-modules
and by $\mathfrak{W}_{\mathfrak{n}}$ the full subcategory of $\mathfrak{F}_{\mathfrak{n}}$ consisting of
all modules $V$ for which there is $i\in\mathbb{N}$ such that $\mathfrak{n}_{i} V=0$. In
\cite[Proposition~3]{BM} it is claimed that every simple finite dimensional $\mathfrak{n}$-module belongs
to $\mathfrak{W}_{\mathfrak{n}}$. This is wrong as shown by the following example:

\begin{example}\label{ex53}
{\rm
Let $\mathfrak{a}$ be any simple finite dimensional Lie algebra and $A:=t\mathbb{C}[t]$ be the associative algebra of polynomials without constant term. Then the Lie algebra $\mathfrak{n}:=A\otimes \mathfrak{a}$ with the Lie
bracket given by $[f\otimes a,g\otimes b]:=fg\otimes[a,b]$ is quasi-nilpotent with
$\mathfrak{n}_i=t^{i+1}\mathbb{C}[t]\otimes \mathfrak{a}$. Any simple $\mathfrak{a}$-module $V$ becomes a
simple $\mathfrak{n}$-module via the evaluation map sending $t$ to $1$.
}
\end{example}

This suggests the following revision of the Whittaker setup from \cite{BM}: Let $\mathfrak{g}$ be a Lie algebra
and $\mathfrak{n}$ a Lie subalgebra of $\mathfrak{g}$. Denote by $\mathfrak{W}_{\mathfrak{n}}^{\mathfrak{g}}$
the full subcategory of the category of $\mathfrak{g}$-modules which consists of all modules $V$ for which
any $v\in V$ belongs to some $\mathfrak{n}$-submodule $X\subset V$ such that $X\in \mathfrak{W}_{\mathfrak{n}}$.
We will say that $(\mathfrak{g},\mathfrak{n})$ is a {\em Whittaker pair} provided that $\mathfrak{n}$ is
quasi-nilpotent and $\mathfrak{g}/\mathfrak{n}\in \mathfrak{W}_{\mathfrak{n}}^{\mathfrak{g}}$. Objects in
the category $\mathfrak{W}_{\mathfrak{n}}^{\mathfrak{g}}$ are called {\em Whittaker modules} for the
Whittaker pair $(\mathfrak{g},\mathfrak{n})$. All general results of \cite{BM} are true if the definition of the
category $\mathfrak{W}_{\mathfrak{n}}^{\mathfrak{g}}$ is adjusted in this way (in \cite{BM} this notation
is used for the category defined in the next paragraph).

It is also natural to consider the full subcategory $\mathfrak{F}_{\mathfrak{n}}^{\mathfrak{g}}$
of the category of $\mathfrak{g}$-modules which consists of all modules $V$ for which
any $v\in V$ belongs to some $\mathfrak{n}$-submodule $X\subset V$ such that $X\in \mathfrak{F}_{\mathfrak{n}}$.
Objects in the category $\mathfrak{F}_{\mathfrak{n}}^{\mathfrak{g}}$ are called {\em generalized Whittaker modules}
for the Whittaker pair $(\mathfrak{g},\mathfrak{n})$.

There are many quasi-nilpotent algebras $\mathfrak{n}$ for which
$\mathfrak{F}_{\mathfrak{n}}=\mathfrak{W}_{\mathfrak{n}}$. For example, this is the case if $\mathfrak{n}$
is finite dimensional (and hence nilpotent). Two other important examples are described below in
Subsection~\ref{s4.2} and \ref{s4.3}. The most important example of $\mathfrak{n}$ for which
$\mathfrak{F}_{\mathfrak{n}}\neq \mathfrak{W}_{\mathfrak{n}}$ is when $\mathfrak{n}$ is the positive part
of the standard triangular  decomposition of an affine Lie algebra $\mathfrak{g}$. Generalized Whittaker
modules in this case are studied in \cite{CGZ}.

\subsection{Whittaker setup for the Virasoro algebra}\label{s4.2}

In the case of the Virasoro algebra $\mathfrak{V}$ we have:

\begin{proposition}\label{lem61}
Let $\mathfrak{n}$ be a Lie subalgebra of $\mathfrak{V}_+^{(0)}$ such that
$\mathfrak{n}\supset \mathfrak{V}_+^{(k)}$ for some $k\in\mathbb{Z}_+$. Then:
\begin{enumerate}[$($a$)$]
\item \label{lem61.1} $\mathfrak{F}_{\mathfrak{n}}=\mathfrak{W}_{\mathfrak{n}}$.
\item \label{lem61.2} $(\mathfrak{V},\mathfrak{n})$ is Whittaker pair.
\end{enumerate}
\end{proposition}

\begin{proof}
Claim \eqref{lem61.1} follows from the first paragraph of Subsection~\ref{s2.1}.
Claim \eqref{lem61.2} follows from definitions.
\end{proof}

From Proposition~\ref{lem61}\eqref{lem61.1} it follows that for any $\mathfrak{n}$ as in Proposition~\ref{lem61}
any simple module in $\mathfrak{W}_{\mathfrak{n}}^{\mathfrak{g}}$ is contained in the list provided by
Theorem~\ref{thmmain2}.

\subsection{Whittaker setup for the Witt algebras}\label{s4.3}

For $n\in\mathbb{N}$ denote by $\mathfrak{w}_n$ the Lie algebra of all derivations of the polynomial algebra
$\mathbb{C}[x_1,x_2,\dots,x_n]$ (the classical {\em Witt algebra}, which is a Lie algebra of Cartan type).
Set $\mathbf{N}:=\{1,2,\dots,n\}$ and for $i\in \mathbf{N}$ let $\partial_i$ denote the derivation
$\frac{\partial}{\partial x_i}$. For $i\in\mathbf{N}$ and
$\mathbf{m}=(m_1,m_2,\dots,m_n)\in\mathbb{Z}_+^n$ the elements
$D_i(\mathbf{m})=x_1^{m_1}x_2^{m_2}\cdots x_n^{m_n}\partial_i$ form the standard basis of
$\mathfrak{w}_n$. The linear span of $x_i\partial_i$, $i\in\mathbf{N}$, is the standard Cartan
subalgebra $\mathfrak{h}$ of $\mathfrak{w}_n$ and the linear span $\mathfrak{a}$ of $x_i\partial_j$,
$i,j\in\mathbf{N}$, is a copy of $\mathfrak{gl}_n$. Let
\begin{displaymath}
\mathfrak{a}=\mathfrak{n}_-^{\mathfrak{a}}\oplus \mathfrak{h}\oplus \mathfrak{n}_+^{\mathfrak{a}}
\end{displaymath}
be some triangular decomposition of $\mathfrak{a}$. Denote by $\mathfrak{n}_{+}$ the linear span of
$\mathfrak{n}_+^{\mathfrak{a}}$ and $\partial_i$, $i\in\mathbf{N}$. Denote by $\mathfrak{n}_{}$
the linear span of $\mathfrak{n}_-^{\mathfrak{a}}$ and all the elements $D_i(\mathbf{m})$ which are contained in
neither $\mathfrak{a}$ nor $\mathfrak{n}_+$. Then we have a decomposition
\begin{displaymath}
\mathfrak{w}_n=\mathfrak{n}_-\oplus \mathfrak{h}\oplus \mathfrak{n}_+
\end{displaymath}
into a direct sum of subalgebras and both $(\mathfrak{w}_n,\mathfrak{n}_-)$ and $(\mathfrak{w}_n,\mathfrak{n}_+)$
are Whittaker pairs (see \cite[5.1]{BM}). The algebra $\mathfrak{n}_+$ is nilpotent finite dimensional, while
$\mathfrak{n}_-$ is neither nor (but it is quasi-nilpotent).

\begin{proposition}\label{prop62}
For $\mathfrak{n}=\mathfrak{n}_-$ we have  $\mathfrak{F}_{\mathfrak{n}}=\mathfrak{W}_{\mathfrak{n}}$.
\end{proposition}

\begin{proof}
Note that the claim of Proposition~\ref{lem61} in the case  of the algebra $\mathfrak{V}_+^{(0)}$ is the claim of
Proposition~\ref{prop62} in the case of the algebra  $\mathfrak{w}_1$. Hence we are left to
consider the case $n>1$.

Let $V$ be a finite dimensional $\mathfrak{n}$-module and $\mathfrak{i}\subset \mathfrak{n}$ be the kernel of the
representation map. To prove our proposition it is enough to show that $\mathfrak{i}$ contains all but finitely many
elements of the standard basis. For a fixed $i\in\mathbf{N}$, the elements $x_i^{s}\partial_i$,
$s>0$, form a copy of $\mathfrak{V}_+$ inside $\mathfrak{w}_n$. Hence Proposition~\ref{lem61} implies that
all but finitely many such elements are in $\mathfrak{i}$. Commuting these elements with $x^2_i\partial_j$, where
$j\neq i$, we get that $\mathfrak{i}$ contains all but finitely many elements of the form
$x_j^{s}\partial_i$. Hence there is $N\in\mathbb{N}$ such that $\mathfrak{i}$ contains all
$x_j^{s}\partial_i$ where $s>N$.

Now we claim that for any $i\in\mathbf{N}$ and any $\mathbf{m}$ such that $m_j>N$ for some $j$ we have
$D_i(\mathbf{m})\in \mathfrak{i}$. Write $\mathbf{m}=\mathbf{m}'+\mathbf{m}''$,
where $m'_j=m_j$ while $m'_s=0$ for all $s\neq j$. Assume first that $i=j$. Then we have
$D_i(\mathbf{m}'+\varepsilon_i)\in \mathfrak{i}$ by the  previous paragraph and the
commutator $[D_i(\mathbf{m}'+\varepsilon_i),D_i(\mathbf{m}'')]$ equals $D_i(\mathbf{m})$ up to a nonzero scalar.
In the case $i\neq j$ we have $D_j(\mathbf{m}')\in \mathfrak{i}$ by the
previous paragraph and the commutator $[D_j(\mathbf{m}'),D_i(\mathbf{m}''+\varepsilon_j)]$
equals $D_i(\mathbf{m})$ up to a nonzero scalar. The claim of the proposition follows.
\end{proof}

\vspace{2mm}

\noindent
V.M.: Department of Mathematics, Uppsala University,
Box 480, SE-751 06, Uppsala, Sweden; e-mail: {\tt mazor\symbol{64}math.uu.se}
\vspace{1mm}

\noindent K.Z.: Department of Mathematics, Wilfrid Laurier
University, Waterloo, Ontario, N2L 3C5, Canada; and College of Mathematics and
Information Science, Hebei Normal (Teachers) University, Shijiazhuang 050016,
Hebei, P. R. China. e-mail:  {\tt kzhao\symbol{64}wlu.ca}

\end{document}